\documentclass{amsart}
\usepackage{amsthm,amsfonts,amssymb,amsmath,amscd,graphics%,picins,amscd
	,newclude, enumitem,geometry,stmaryrd, fancyhdr
}
\usepackage{textcomp}
\usepackage{appendix}
\usepackage{xcolor}
\usepackage{longtable}
\usepackage{url}
\usepackage[centerlast]{subfigure}
\usepackage[all,2cell, cmtip]{xy} \UseAllTwocells
\begin{document}
%	\begin{titlepage}
%		\begin{center}
%			\textbf{\uppercase{Compact closed categories and $\Gamma$-categories }}\\[0.015cm]
%			\textit{(with an appendix by André Joyal)} \\[1cm]
%			\Large{\uppercase{Amit Sharma}}\\[0.25cm]
%		
%		\end{center}
%
%		ABSTRACT: In this paper we extend the notion of compact closed categories to coherently commutative monoidal categories. We construct a model category of (permutative) compact closed categories and a model category of coherently compact closed categories which are manifestations of the aforementioned extension. We show that the Segal's nerve functor is a Quillen equivalence between the two model categories. The construction of a model category of coherently compact closed categories leads to a proof of the one dimensional cobordism hypothesis based on a purely homotopical algebra argument.
%		% is to show that the tensor product of symmetric monoidal categories restricts to a tensor product of compact closed categories.
%		%We discuss the enrichment of dagger compact closed categories.

%	\end{titlepage}

	\title[On Cofibrations of Permutative categories]{On Cofibrations of Permutative categories} 
%	(with an appendix by André Joyal)	
%}
% %\subtitle{}
	\author[A. Sharma]{Amit Sharma}
	\email{asharm24@kent.edu}
	\address {Department of mathematical sciences\\ Kent State University\\
		Kent,OH}
%%	\renewcommand{\maketitle}{\centering }
%%	\subtitle{}
%% \thanks{(with an appendix by André Joyal)}
%	\date{\today}
\begin{abstract}
	In this note we introduce a notion of free cofibrations of permutative categories. We show that each cofibration of permutative categories is a retract of a free cofibration.
	\end{abstract}
\date{\today}
%%%%%%%%%%%%%%%%%%%%%%%%%%%%%%%%%%%%%%%%%%%%%%%%%%%%%%%%%%%%%%%%%%%%%%%%%%%%%%%%
% Custom definitions
%%%%%%%%%%%%%%%%%%%%%%%%%%%%%%%%%%%%%%%%%%%%%%%%%%%%%%%%%%%%%%%%%%%%%%%%%%%%%%%%
% \newcommand{\C}{{\mathrm{c}}}
% 
% \newcommand{\COS}{{\mathrm{cos}}}
% \newcommand{\SIN}{{\mathrm{sin}}}
% \newcommand{\DOF}{{\mathrm{dof}}}
% \newcommand{\MSEC}{\mu{\mathrm{s}}}
% \newcommand{\NS}{\mathrm{ns}}
% \newcommand{\PPM}{{\mathrm{ppm}}}
% \newcommand{\ENR}{{\mathrm{GeV}}}
% \newcommand{\MOM}{{\mathrm{GeV/c}}}

\newcommand{\CONT}{\noindent}
\newcommand{\FIG}{Fig.\ }
\newcommand{\FIGS}{Figs.\ }
\newcommand{\SEC}{Sec.\ }
\newcommand{\SECS}{Secs.\ }
\newcommand{\TAB}{Table }
\newcommand{\TABS}{Tables }
\newcommand{\EQ}{Eq.\ }
\newcommand{\EQS}{Eqs.\ }
\newcommand{\APP}{Appendix }
\newcommand{\APPS}{Appendices }
\newcommand{\CHP}{Chapter }
\newcommand{\CHPS}{Chapters }

\newcommand{\OFF}{\emph{G2off}~}
\newcommand{\TOO}{\emph{G2Too}~}
%%%%%%%%%%%%%%%%%%%%%%%%%%%%%%%%%%%%%%%%%%%%%%%%%%%%%%%%%%%%%%%%%%%%%%%%%%%%%%%%
\newcommand{\CatS}{Cat_{\bigS}}
\newcommand{\PicS}{{\underline{\pic}}^{\oplus}}
%%%%%%%%%%%%%%%%%%%%%%%%%%%%%%%%%%%%%%%%%%%%%%%%%%%%%%%%%%%%%%%%%%%%%%%%%%%%%%%%
\newcommand{\HPicS}{{Hom^{\oplus}_{\pic}}}

\newtheorem{thm}{Theorem}[section]
\newtheorem{lem}[thm]{Lemma}
\newtheorem{conj}[thm]{Conjecture}
\newtheorem{coro}[thm]{Corollary}
\newtheorem{prop}[thm]{Proposition}

\theoremstyle{definition}
\newtheorem{df}[thm]{Definition}
\newtheorem{nota}[thm]{Notation}

\newtheorem{ex}[thm]{Example}
\newtheorem{exs}[thm]{Examples}

\theoremstyle{remark}
\newtheorem*{note}{Note}
\newtheorem{rem}{Remark}
\newtheorem{ack}{Acknowledgments}

\newcommand{\ChI}{{\textit{\v C}}\textit{ech}}
\newcommand{\Ch}{{\v C}ech}

\newcommand{\ChZG}{hermitian line $0$-gerbe}
\renewcommand{\theack}{$\! \! \!$}

\newcommand{\ChG}{flat hermitian line $1$-gerbe}
\newcommand{\ChC}{hermitian line $1$-cocycle}
\newcommand{\ChGG}{flat hermitian line $2$-gerbe}
\newcommand{\ChCC}{hermitian line $2$-cocycle}
\newcommand{\id}{id}
\newcommand{\LC}{\mathfrak{C}}
\newcommand{\Coker}{Coker}
\newcommand{\Com}{Com}
\newcommand{\Hom}{Hom}
\newcommand{\Mor}{Mor}
\newcommand{\Map}{Map}
\newcommand{\alg}{alg}
\newcommand{\an}{an}
\newcommand{\Ker}{Ker}
\newcommand{\Ob}{Ob}
\newcommand{\Proj}{\mathbf{Proj}}
\newcommand{\topo}{\mathbf{Top}}
\newcommand{\kan}{\mathcal{K}}
\newcommand{\pkan}{\mathcal{K}_\bullet}
\newcommand{\Kan}{\mathbf{Kan}}
\newcommand{\pKan}{\mathbf{Kan}_\bullet}
\newcommand{\QCat}{\mathbf{QCat}}
\newcommand{\gp}{\mathcal{A}_\infty}
\newcommand{\mdl}{\mathcal{M}\textit{odel}}
\newcommand{\sSets}{\mathbf{sSets}}
\newcommand{\sSetsQ}{(\mathbf{sSets, Q})}
\newcommand{\sSetsK}{(\mathbf{sSets, \Kan})}
\newcommand{\pSSets}{\mathbf{sSets}_\bullet}
\newcommand{\pSSetsK}{(\mathbf{sSets}_\bullet, \Kan)}
\newcommand{\pSSetsQ}{(\mathbf{sSets_\bullet, Q})}
\newcommand{\cyl}{\mathbf{Cyl}}
\newcommand{\lin}{\mathcal{L}_\infty}
\newcommand{\Vect}{\mathbf{Vect}}
\newcommand{\Aut}{Aut}
\newcommand{\pic}{\mathcal{P}\textit{ic}}
\newcommand{\Dlin}{\pic}
\newcommand{\bigS}{\mathbf{S}}
\newcommand{\bigA}{\mathbf{A}}
\newcommand{\bhom}{\mathbf{hom}}
\newcommand{\bhomK}{\mathbf{hom}({\textit{K}}^+,\textit{-})}
\newcommand{\Bhom}{\mathbf{Hom}}
\newcommand{\bhomk}{\mathbf{hom}^{{\textit{k}}^+}}
\newcommand{\Dlino}{\pic^{\textit{op}}}
\newcommand{\lino}{\mathcal{L}^{\textit{op}}_\infty}
\newcommand{\lind}{\mathcal{L}^\delta_\infty}
\newcommand{\linK}{\mathcal{L}_\infty(\kan)}
\newcommand{\linC}{\mathcal{L}_\infty\text{-category}}
\newcommand{\linCs}{\mathcal{L}_\infty\text{-categories}}
\newcommand{\ainCs}{\text{additive} \ \infty-\text{categories}}
\newcommand{\ainC}{\text{additive} \ \infty-\text{category}}
\newcommand{\inC}{\infty\text{-category}}
\newcommand{\inCs}{\infty\text{-categories}}
\newcommand{\gS}{{\Gamma}\text{-space}}
\newcommand{\gSet}{{\Gamma}\text{-set}}
\newcommand{\ggS}{\Gamma \times \Gamma\text{-space}}
\newcommand{\gSs}{\Gamma\text{-spaces}}
\newcommand{\gSets}{\Gamma\text{-sets}}
\newcommand{\ggSs}{\Gamma \times \Gamma\text{-spaces}}
\newcommand{\gO}{\Gamma-\text{object}}
\newcommand{\gSCat}{{\Gamma}\text{-space category}}
\newcommand{\pss}{\mathbf{S}_\bullet}
\newcommand{\gSC}{{{{\Gamma}}\mathcal{S}}}
\newcommand{\pGSC}{{{{\Gamma}}\mathcal{S}}_\bullet}
\newcommand{\pGSCStr}{{{{\Gamma}}\mathcal{S}}_\bullet^{\textit{str}}}
\newcommand{\ggSC}{{\Gamma\Gamma\mathcal{S}}}
\newcommand{\gSD}{\mathbf{D}(\gSC^{\textit{f}})}
\newcommand{\sCat}{\mathbf{sCat}}
\newcommand{\pSCat}{\mathbf{sCat}_\bullet}
\newcommand{\gSetCat}{{{{\Gamma}}\mathcal{S}\textit{et}}}
\newcommand{\Dhom}{\mathbf{R}Hom_{\pic}}
\newcommand{\gop}{\Gamma^{\textit{op}}}
\newcommand{\fU}{\mathbf{U}}
\newcommand{\cDN}{\underset{\mathbf{D}[\textit{n}^+]}{\circ}}
\newcommand{\cDK}{\underset{\mathbf{D}[\textit{k}^+]}{\circ}}
\newcommand{\cDL}{\underset{\mathbf{D}[\textit{l}^+]}{\circ}}
\newcommand{\cD}{\underset{\gSD}{\circ}}
\newcommand{\cDT}{\underset{\gSD}{\widetilde{\circ}}}
\newcommand{\ppsSets}{\sSets_{\bullet, \bullet}}
\newcommand{\gdHom}{\underline{Hom}_{\gSD}}
\newcommand{\HomU}{\underline{Hom}}
\newcommand{\ominf}{\Omega_\infty}
\newcommand{\ev}{ev}
\newcommand{\cu}{C(X;\mathfrak{U}_I)}
\newcommand{\Sing}{Sing}
\newcommand{\AlgEin}{\A\textit{lg}_{\E_\infty}}
\newcommand{\SFunc}[2]{\mathbf{SFunc}({#1} ; {#2})}
\newcommand{\unit}[1]{\mathrm{1}_{#1}}
\newcommand{\liminj}{\varinjlim}
\newcommand{\limproj}{\varprojlim}
\newcommand{\HMapC}[3]{\mathcal{M}\textit{ap}^{\textit{h}}_{#3}(#1, #2)}
\newcommand{\tensPGSR}[2]{#1 \underset{\gSR}\wedge #2}
\newcommand{\pTensP}[3]{#1 \underset{#3}\wedge #2}
\newcommand{\MGCat}[2]{\underline{\map}_{\gCAT}({#1},{ #2})}
\newcommand{\MGCatGen}[3]{\underline{\map}_{#3}({#1},{ #2})}
\newcommand{\MGBoxCat}[2]{\underline{\map}_{\gSC}^{\Box}({#1},{ #2})}
\newcommand{\TensPFunc}[1]{- \underset{#1} \otimes -}
\newcommand{\TensP}[3]{#1 \underset{#3}\otimes #2}
\newcommand{\MapC}[3]{\mathcal{M}\textit{ap}_{#3}(#1, #2)}
\newcommand{\bHom}[3]{{#2}^{#1}}
\newcommand{\gn}[1]{\Gamma^{#1}}
\newcommand{\gnk}[2]{\Gamma^{#1}({#2}^+)}
\newcommand{\gnf}[2]{\Gamma^{#1}({#2})}
\newcommand{\ggn}[1]{\Gamma\Gamma^{#1}}
\newcommand{\Nat}{\mathbb{N}}
\newcommand{\partition}[2]{\delta^{#1}_{#2}}
\newcommand{\inclusion}[2]{\iota^{#1}_{#2}}
\newcommand{\EinQC}{\text{coherently commutative monoidal quasi-category}} 
\newcommand{\EinQCs}{\text{coherently commutative monoidal quasi-categories}}
\newcommand{\pHomCat}[2]{[#1,#2]_{\bullet}}
\newcommand{\CatHom}[3]{[#1,#2]^{#3}}
\newcommand{\pCatHom}[3]{[#1,#2]_\bullet^{#3}}
\newcommand{\EinC}{\text{coherently commutative monoidal category}}
\newcommand{\EinCs}{\text{coherently commutative monoidal categories}}
\newcommand{\EinLO}{E_\infty{\text{- local object}}}
\newcommand{\EinSLO}{\E_\infty\S{\text{- local object}}}
\newcommand{\Ein}{E_\infty}
\newcommand{\EinS}{E_\infty{\text{- space}}}
\newcommand{\EinSs}{E_\infty{\text{- spaces}}}
\newcommand{\PCat}{\mathbf{Perm}}
\newcommand{\PCatCC}{\mathbf{Perm}^{\textit{cc}}}
\newcommand{\nor}[1]{{#1}^\textit{nor}}
\newcommand{\pSSetsHom}[3]{[#1,#2]_\bullet^{#3}}
\newcommand{\PNat}{\overline{\L}}
\newcommand{\PStr}{\L}
\newcommand{\Gn}[1]{\Gamma[#1]}
\newcommand{\GIH}{\Gamma\textit{H}_{\textit{in}}}
\newcommand{\QStr}[1]{\L_\bullet(\ud{#1})}
\newcommand{\QStrF}{\L_\bullet}
\newcommand{\Kbar}{\overline{\K}}
\newcommand{\gPerm}{{\Gamma\PCat}}
\newcommand{\gCat}{{\Gamma{\text{-category}}}}
\newcommand{\Du}{{\mathfrak{D}}}
\newcommand{\EHom}[2]{[#1; #2]^\E}
\newcommand{\Cob}[1]{{\mathfrak{Cob}}^{\textit{fr}}(\ud{#1})}
\newcommand{\gpd}{{\mathbf{Gpd}}}
\newcommand{\ol}[1]{\overline{#1}}
\newcommand{\gCob}[1]{\Gamma\Cob{#1}}
\newcommand{\gCatCC}{{\gCAT^{\textit{cc}}}}
\newcommand{\gCatSM}{{\gCAT^{\otimes}}}
\newcommand{\gPCat}{\Gamma\PCat}
\newcommand{\pGPCat}{\gPCat_\bullet}
\newcommand{\lCC}{\PStr^{\textit{cc}}}
\newcommand{\kCC}{\K^{\textit{cc}}}

\def\F{\mathcal F}
\def\Pic{\mathbf{2}\mathcal P\textit{ic}}
\def\nc{\mathbb C}

\def\Z{\mathbb Z}
\def\P{\mathbb P}
\def\J{\mathcal J}
\def\I{\mathcal I}
\def\nC{\mathbb C}
\def\H{\mathcal H}
\def\A{\mathcal A}
\def\C{\mathcal C}
\def\D{\mathcal D}
\def\E{\mathcal E}
\def\G{\mathcal G}
\def\B{\mathcal B}
\def\L{\mathcal L}
\def\U{\mathcal U}
\def\K{\mathcal K}

\def\M{\mathcal M}
\def\O{\mathcal O}
\def\R{\mathcal R}
\def\S{\mathcal S}
\def\N{\mathcal N}

\newcommand{\undertilde}[1]{\underset{\sim}{#1}}
\newcommand{\abs}[1]{{\lvert#1\rvert}}
\newcommand{\mC}[1]{\mathfrak{C}(#1)}
\newcommand{\sigInf}[1]{\Sigma^{\infty}{#1}}
\newcommand{\x}[4]{\underset{#1, #2}{ \overset{#3, #4} \prod }}
\newcommand{\mA}[2]{\textit{Add}^n_{#1, #2}}
\newcommand{\mAK}[2]{\textit{Add}^k_{#1, #2}}
\newcommand{\mAL}[2]{\textit{Add}^l_{#1, #2}}
\newcommand{\Mdl}[2]{\L_\infty}% \left(#1, #2 \right)}
\newcommand{\inv}[1]{#1^{-1}}
\newcommand{\Lan}[2]{\mathbf{Lan}_{#1}(#2)}
\newcommand{\ccCat}{\mathbf{cc}\PCat}
\newcommand{\fCC}{\F^{\textit{cc}}}

\newcommand{\del}{\partial}
\newcommand{\sCatO}{\mathcal{S}Cat_\O}
\newcommand{\FCgop}{\mathbf{F}\mC{N(\gop)}}
\newcommand{\hProd}{{\overset{h} \oplus}}
\newcommand{\hProdn}{\underset{n}{\overset{h} \oplus}}
\newcommand{\hProdk}[1]{\underset{#1}{\overset{h} \oplus}}
\newcommand{\map}{\mathcal{M}\textit{ap}}
\newcommand{\SMGS}[2]{\map_{\gSC}({#1},{ #2})}
\newcommand{\MGS}[2]{\underline{\map}_{\gSC}({#1},{ #2})}
\newcommand{\MGSBox}[2]{\underline{\map}^{\Box}_{\gSC}({#1},{ #2})}
\newcommand{\Aqcat}[1]{\underline{#1}^\oplus}
\newcommand{\Cat}{\mathbf{Cat}}
\newcommand{\Sp}{\mathbf{Sp}}
\newcommand{\SpStb}{\mathbf{Sp}^{\textit{stable}}}
\newcommand{\SpStr}{\mathbf{Sp}^{\textit{strict}}}
\newcommand{\Sspec}{\mathbb{S}}
\newcommand{\ud}[1]{\underline{#1}}
\newcommand{\inrt}{\mathbf{Inrt}}
\newcommand{\act}{\mathbf{Act}}
\newcommand{\StrSMHom}[2]{[#1,#2]_\otimes^{\textit{str}}}
\newcommand{\SMHom}[2]{[#1,#2]_\otimes}
\newcommand{\ESMHom}[2]{[#1,#2]_\otimes^\E}
\newcommand{\gCats}{\Gamma\text{-categories}}
\newcommand{\gCAT}{\Gamma\Cat}
\newcommand{\gCATCCM}{\Gamma\Cat^\otimes}
\newcommand{\KStr}[1]{\K(#1)}
\newcommand{\pGPCatStr}{\pGPCat^{\textit{Str}}}

	\maketitle
	%	\tableofcontents

%	\include{Introduction}
%	\include{ccPCat}
%	\include{CCPerm}
%	\include{CCGCat}
	
	% \include{BordCat}
 % \input{GCFunctor}
  % \input{TensPCCCat}
	
%	
%	\include{Duality}
%	\include{LocMdlCats}
%   \include{TransferMdlStr}
\section{Introduction}
A \emph{permutative} category is a symmetric monoidal category whose associativity and unit natural isomorphisms are identites. Permutative categories have generated significant interest in topology. An infinite loop space machine was constructed on permutative categories in \cite{May4}.
A $K$-theory (multi-)functor from a \emph{multicategory} of permutative categories into a symmetric monoidal category of symmetric spectra, which preserves the multiplicative structure, was constructed in \cite{Elmendorf2006RingsMA}. In \cite{mandell2}, the  $K$-theory of \cite{Elmendorf2006RingsMA} was enhanced to a lax symmetric monoidal functor. It was shown in \cite{mandell} that permutative categories model connective spectra.

 Every symmetric monoidal category is equivalent (by a symmetric monoidal functor) to a permutative category. The category of symmetric monoidal categopries $\mathbf{SMCAT}$ does NOT have a model category structure, however its subcategory of permutative categories and strict symmetric monoidal functors $\PCat$ can be endowded with a model category structure. The category $\PCat$ is isomorphic to the category of algebras over the (categorical) Barrat-Eccles operad. Using this fact, the a model category structure follows from \cite{BM1} and \cite[Thm. 4.5]{Lack}. This model category structure is called the natural model category structure of permutative categories.
 
 The main objective of this note is to identify a class of cofibrations in the natural model category $\PCat$, called \emph{free} cofibrations such that every cofibration in $\PCat$ is a retract of a free cofibration. A useful property of free cofibrations is that cobase changes along a free cofibration preserve acyclic fibrations in the natural model category $\PCat$.
 This property allows us to prove that the natural model category $\PCat$ is \emph{left proper}.
 
 \begin{ack}
 	The author is thankful to Andre Joyal for proposing the idea of a free-cofibration and also for many insightful discussions regarding this note.
 	\end{ack}
   \section{Cofibrations in $\PCat$ and left properness}

In this note we define a class of maps called \emph{free} cofibrations in the natural model category of permutative categories $\PCat$. We show that a strict symmetric monoidal functor is a cofibration in $\PCat$ if and only if it is a retract of a free cofibration.
Using this characterization of cofibrations we will show that the natural model category $\PCat$ is left proper. A characterization of  cofibrations in $\PCat$ was formulated, purely in terms of object functions (which are monoid homomorphisms) of the underlying strict symmetric monoidal functor, in \cite{Sharma}. In order to define free cofibrations, we will start by reviewing some basic notions of permutative categories:
\begin{df}
	\label{perm-cat}
	A symmetric monoidal category is called a \emph{permutative} category or a strict symmetric monoidal category if it is strictly associative and strictly unital.
	\end{df}
\begin{rem}
	A permutative category is an internal category in the category of monoids.
	\end{rem}
 We recall that the forgetful functor $U:\PCat \to \Cat$ has a left adjoint $\F:\Cat \to \PCat$.
 \begin{df}
 	A monoid $M$ is called a \emph{free} monoid if there exists a (dotted) lifting monoid homomorphism whenever we have the following (outer) commutative diagram of monoid homomorphisms:
 	\[
 	\xymatrix{
 	\ast \ar[r] \ar[d] & N \ar[d]^p \\
 	M \ar[r] \ar@{-->}[ru] & Q
 	}
 	\]
 	where $p$ is a surjective monoid homomorphism and $\ast$ is a \emph{zero} object in the  category of monoids.
 	\end{df}
\begin{df}
\label{Special-Cof}
A \emph{free cofibration} of permutative categories is a (strict symmetric monoidal) functor $i:A \to C$ whose object function is the inclusion $Ob(i):Ob(A) \to Ob(A) \vee M = Ob(C)$, where $M$ is a free monoid and the coproduct is taken in the category of monoids.
\end{df}

The next proposition presents the desired characterization of cofibrations:
\begin{prop}
\label{char-cof}
A strict symmetric monoidal functor $F:C \to D$ is a cofibration in $\PCat$ if and only if it is a retract of a free cofibration by a map that fixes $C$.
\end{prop}
\begin{proof}
	Let us first assume that $F$ is a retract of a free cofibration $i:E \to M$. We observe that the object function of a free cofibration has the left lifting property with respect to all surjective monoid homomorphisms, therefore each free cofibration is a cofibration in $\PCat$. A retract of a cofibration is again a cofibration. Thus, $F$ is a cofibration in $\PCat$.
	
	Conversely, let us assume that $F$ is a cofibration in $\PCat$.
	We have the following (outer) commutative diagram in the category of monoids
	\begin{equation*}
	\xymatrix{
	Ob(C) \ar[r]^{i \ \ \ \ \ \ \ } \ar[d]_{Ob(F)}  & Ob(C) \vee \F_m(Ob(D)) \ar[d]^p \\
	Ob(D) \ar@{=}[r] \ar@{-->}[ru]_L & Ob(D)
     }
	\end{equation*}
	where $\F_m(Ob(D))$ is the free monoid generated by the set $Ob(D)$, $i$ is the inclusion into the coproduct and $p = Ob(F) \vee \epsilon$, the summand $\epsilon:\F_m(Ob(D)) \to Ob(D)$ is the counit of the reflection:
	\[
	\F_m:\mathbf{Set} \rightleftarrows \mathbf{Mon}:U
	\] 
	Since the right vertical homomorphism of monoids is surjective and $F$ is a cofibration by assumption, therefore there exists a (dotted) lifting homomorphism $L$ which makes the whole diagram commutative. Thus $Ob(F)$ is a retract of the inclusion $i$ in the category of monoids. We will construct a strict symmetric monoidal functor $I:C \to E$ whose object function is the inclusion $i$ and show that $F$ is a retract of $I$. We begin by constructing the category $E$:
	
	The object set of $E$ is $Ob(C) \vee F(Ob(D))$.
	The morphism monoid of $E$ is defined by the following pullback square in the category of monoids:
	\begin{equation}
	\label{Mor-E}
	\xymatrix{
	Mor(E) \ar[r]^{p_1} \ar[d]_{p_2} & Mor(D) \ar[d]^{(s_D,t_D)} \\
	(Ob(C) \vee F(Ob(D))) \times (Ob(C) \vee F(Ob(D))) \ar[r]_{ \ \ \ \ \ \ \ \ \ \ \ \ \ \ \ \ \ \ p \times p} & Ob(D) \times Ob(D)
    }
	\end{equation}
	
	We will denote the projection map $p_2$ in the above cartesian square by $(s_E,t_E)$. This pair will be source and target maps for the proposed category $E$. The projection map $p_1$ in the above cartesian diagram restricts to a map between the set of composable arrows in $E$ and $D$: 
	\[
	p_1^c:Mor(E) \underset{s_E=t_E} \times Mor(E) \to Mor(D) \underset{s_D=t_D} \times Mor(D).
	\]
	Now we observe that the composite $(-\underset{D} \circ -) \circ p_1^c$ factors through $Mor(E)$ as follows:
	\begin{equation}
	\label{Comp-E}
	\xymatrix{
		Mor(E) \underset{s_E=t_E} \times Mor(E) \ar[r]^{{ \ \ \ \ \ \ \ \  -\underset{E} \circ -}} \ar[d]_{p_1^c} & Mor(E) \ar[d]^{p_1} \\
		Mor(D) \underset{s_D=t_D} \times Mor(D) \ar[r]_{ \ \ \ \ \ \ \ \  -\underset{D} \circ -} & Mor(D)
	}
	\end{equation}
	The map $-\underset{E} \circ -$ in the above commutative diagram provides the composition of the category $E$. Finally, we define the symmetry natural transformation of $E$ as follows:
	\begin{equation}
	\label{symmetry-E}
	\gamma^E_{z_1,z_2} := \gamma^D_{p(z_1), p(z_2)}
	\end{equation}
	for each pair of objects $z_1, z_2 \in Ob(E)$.
	This defined a permutative category $(E, - \underset{E} \boxtimes -, \gamma^E)$, where the tensor product is uniquely determined by the monoid structures on $Ob(E)$ and $Mor(E)$.
	
	The commutative diagrams \eqref{Mor-E}, \eqref{Comp-E} and the definition of the symmetry natural transformation \eqref{symmetry-E} together imply that there is a strict symmetric monoidal functor $P:E \to D$ whose object homomorphism is $p$ and morphism homomorphism is $p_1$. Further $P$ is surjective on objects and also fully-faithful. This implies that $P$ is an acyclic fibration in the natural model category $\PCat$.
	
	Now we construct the free cofibration $I:C \to E$ mentioned above. The object homomorphism of $I$ is the inclusion $i:Ob(C) \to Ob(C) \vee F(Ob(D))$. The morphism homomorphism of $I$ is defined as follows:
	\[
	Mor(I) := Mor(F).
	\]
	In other words, $I(f) = F(f)$ for each morphism $f \in Mor(C)$. Now we have the following (outer) commutative diagram in $\PCat$:
	\begin{equation*}
	\xymatrix{
		C \ar[r]^{I  } \ar[d]_{F}  & E \ar[d]^P \\
		D \ar@{=}[r] \ar@{-->}[ru]_L & D
	}
	\end{equation*}
	
	Since $F$ is a cofibration and $P$ is an acyclic fibration in the natural model category $\PCat$, therefore there exists a (dotted) lifting arrow $L$ which makes the entire diagram commutative. This implies that $F$ is a retract of the free cofibration $I$ in the natural model category $\PCat$.
\end{proof}

z
   \section[Left properness of $\PCat$]{Left properness of the natural model category $\PCat$}
\label{left-prop}

In this section we show that the natural model category of permutative categories $\PCat$ is left proper. We recall that a model category is left proper if the cobase change of a weak-equivalence along a cofibration is again a weak-equivalence. We will first show that the cobase change of a weak-equivalence along a free cofibration is a weak-equivalence. Using this intermediate result, we will prove the left properness of $\PCat$.

Let  $G:A \to B$ be an acyclic fibration in $\PCat$ and $i_A:A \to C$ be a free cofibration therefore the object monoid of $C$ can be written as a coproduct $Ob(A) \vee V$, where $V$ is a free monoid. We observe that the following commutative square is coCartesian:
\begin{equation*}
\xymatrix{
	Ob(A) \ar[r]^{Ob(i_A) \ \ \ \ } \ar[d]_{Ob(G)} & Ob(A) \vee V \ar[d]^{Ob(G) \vee id} \\
	Ob(B) \ar[r]_{i_B \ \ \ \ } & F(B) \vee V
}	
\end{equation*} 
We will construct the following pushout square in $\PCat$:
\begin{equation*}
\xymatrix{
	A \ar[r]^{i_A \ \ \ \ } \ar[d]_{G} & C \ar[d] \\
	B \ar[r] & B \underset{A}\sqcup C
}	
\end{equation*} 
%	We observe that the object function of the pushout of $G$ along $i_A$ is the monoid homomorphism $Ob(G) \vee id$. We will provide a construction of the pushout permutative category  $B \underset{A}\sqcup C$.

% We already know the object monoid of this permutative category.

A strict symmetric monoidal functor $G:A \to B$ is an acyclic fibration in $\PCat$ if and only if there exists a unital symmetric monoidal section \cite[Cor. 3.5(3)]{Sharma} $S:D \to C$ such that $GS = id_D$ and a monoidal natural isomorphism $\epsilon_S:SG \cong id$.
% Further, the monoidal natural isomorphism $\epsilon_S$ maybe so chosen that the composite in the following diagram is the identity natural transformation:
%\begin{equation}
%\label{str-on-codom-nat-iso}
%\xymatrix{
%	& &  & \\
%B \ar[r]_S & A \ar@/^3pc/[rr]^{id} \ar[r]_G & B \ar@{=>}[u]^{\epsilon_S } \ar[r]_S & A 
%}
%\end{equation}
%In other words,  $\epsilon_S(S(b)) = id_{S(b)}$ for all $b \in B$.
 Let us fix such a section $S:B \to A$ and natural isomorphism $\epsilon_S$.
 \begin{rem}
 	The above characterization of acyclic fibrations implies that $S:B \to A$ is a left-adjoint-right-inverse of $G:A \to B$. This means that $\epsilon_S:SG \cong id_A$ is a counit of an adjoint equivalence whose unit $\eta:GS = id_B$ is the identity natural transformation. This further implies that $G\epsilon_S \cdot \eta G = id_G$.
 	In other words, for each $a \in A$, we have the following equality:
 	\[
 	G(\epsilon_S(a)) \circ \eta(G(a)) = id_{G(a)}.
 	\]
 	Since the unit natural transformation $\eta$ is the identity, therefore $G\epsilon_S = G$.
 	\end{rem}
 \begin{rem}
 	\label{lambda-to-id}
 	Let $b_1, b_2$ be a pair of objects in $B$. Since $\epsilon_S$ is a monoidal natural transformation, therefore we have the following commutative diagram:
 	\begin{equation*}
 	\xymatrix@C=20mm{
 	SG(S(b_1) \otimes S(b_2)) \ar[r]^{\ \ \ \ \epsilon_S(S(b_1), S(b_2))} \ar[d]_{\lambda^S(b_1, b_2)} & S(b_1) \otimes S(b_2) \ar@{=}[d] \\
 	S(GS(b_1) \otimes GS(b_2)) \ar@{=}[r] & S(b_1) \otimes S(b_2)
    }
 	\end{equation*}
 	Thus we have shown that 
 	\[
 	\lambda^S = \epsilon_S S.
 	\]
 	This further implies that
 	\[
 	G\lambda^S = G\epsilon_S S = GS = id_B.
 	\]
 	\end{rem}
 
  The unital symmetric monoidal functor $S$ gives us the following unital symmetric monoidal functor:
\[
S \vee \F(C; V):B \vee \F(C; V) \to A \vee \F(C; V),
\]
where $\F(C; V)$ is the full permutative subcategory of $C$ whose object set is the (free) monoid $V$. We observe that $S \vee \F(C; V)$ is a section of the strict symmetric monoidal functor $G \vee \F(C; V)$ \emph{i.e.} $(G \vee \F(C; V)) \circ (S \vee \F(C; V)) = id$. Moreover, we get a monoidal natural isomorphism 
\[
\epsilon_S \vee \F(C;V):(S \vee \F(C; V)) \circ (G \vee \F(C; V)) \cong id
\]
Hence the functor $G \vee \F(C; V)$ is an acyclic fibration in the natural model category $\PCat$ by \cite[Cor. 3.5(3)]{Sharma}.

We observe that the free cofibration $i_A$ factors as follows:
\begin{equation}
\label{fact-inc}
\xymatrix{
A \ar[rr] \ar[rd]_{\iota_A} && C \\
& A \vee \F(C;V) \ar[ru]_{i_{A,V}}
 }
\end{equation}
where $\iota_A:A \to A \vee \F(C;V)$ is the inclusion into the coproduct and $i_{A,V}:A \vee \F(C;V) \to C$ is the unique map induced by the inclusions $i_A:A \to C$ and $i_V:\F(C,V) \to C$

\begin{rem}
	\label{co-Cart-1}
The following commutative square is a coCartesian:

\begin{equation}
\xymatrix{
	A \ar[r]^{\iota_A \ \ \ \ \ \ } \ar[d]_{G} & A \vee \F(C;V) \ar[d]^{G \vee \F(C;V)} \\
 B \ar[r]_{\iota_B \ \ \ \ \ \ }	&  B \vee \F(C;V)
}
\end{equation}
\end{rem}
We observe that the object monoid of $C$ is the same as the object monoid of $A \vee \F(C;V)$, namely the coproduct $(Ob(A)) \vee V$. This implies that for each $c \in Ob(C)$ there is the following isomorphism in $C$:
\[
(i_{A,V} \circ (\epsilon_S \vee \F(C; V)))(c):(S \vee \F(C; V)) \circ (G \vee \F(C; V))(c) \cong c,
\]
 
Now it follows from \cite[Prop. 2.7]{Sharma} that there exists a (uniquely defined) functor $S_C:C \to C$ and a natural isomorphism
$\delta^C:id_C \cong S_C$. The functor $S_C$ is defined on objects as follows:
\[
S_C(c) := (S \vee \F(C; V)) \circ (G \vee \F(C; V))(c). 
\]
 The following lemma now tells us that $S_C$ is a unital symmetric monoidal functor and $\delta^C$ is a monoidal natural isomorphism:
 \begin{lem}
 	\label{SM-Func-closed-isom}
 	Given a unital oplax symmetric monoidal functor $(F, \lambda_F)$ between two symmetric monoidal categories $C$ and $D$, a functor $G:C \to D$, and a unital natural isomorphism
 	$\alpha:F \cong G$, there is a unique natural isomorphism $\lambda_G$ which enhances $G$ to a unital oplax symmetric monoidal functor $(G, \lambda_G)$ such that $\alpha$ is a monoidal natural isomorphism. If $(F, \lambda_F)$ is unital symmetric monoidal then so is $(G, \lambda_G)$.
 \end{lem}
 \begin{proof}
 	We consider the following diagram:
 	\begin{equation*}
 	\label{unit-counit-SMM}
 	\xymatrix@C=2mm{
 		&& C \times C \ar@/^2pc/[dd]^{F \times F} \ar@/_2pc/[dd]_{G \times G}\ar[rrrrrrrr]^{\TensPFunc{C}} &&&&&&&&C \ar@/^2pc/[dd]^{G} \ar@/_2pc/[dd]_{F } &&  \\
 		&&& \ar@{=>}[ll]_{\ \alpha \times \alpha} &&&&&& \ar@{=>}[rr]^{\ \alpha} &&&& \\
 		&&D \times D \ar[rrrrrrrr]_{\TensPFunc{D}} &&&&&&&& D &&
 	}
 	\end{equation*}
 	This diagram helps us define a composite natural isomorphism $\lambda_{G}:G \circ (\TensPFunc{C}) \Rightarrow (\TensPFunc{D}) \circ G \times G$ as follows:
 	\begin{equation}
 	\label{def-OL-str-G}
 	\lambda_{G} := (id_{\TensPFunc{D}} \circ \alpha \times \alpha) \cdot \lambda_F \cdot (\inv{\alpha} \circ id_{\TensPFunc{C}}).
 	\end{equation}
 	This composite natural isomorphism is the unique natural isomorphism which makes $\alpha$ a unital monoidal natural isomorphism.
 	Now we have to check that $\lambda_G$ is a unital monoidal natural isomorphism with respect to the above definition. Clearly, $\lambda_G$ is unital because both $\alpha$ and $\lambda_F$ are unital natural isomorphisms. We first check
 	the symmetry condition \cite[Defn. 2.4 OL. 2]{Sharma}. This condition is satisfied because the following composite diagram commutes
 	\begin{equation*}
 	\label{symm-cond-alpha}
 	\xymatrix@C=12mm{
 		G(\TensP{c_1}{c_2}{C}) \ar[r]^{\inv{\alpha}(\TensP{c_1}{c_2}{C})} \ar[d]_{G(\gamma^C(c_1, c_2))} & F(\TensP{c_1}{c_2}{C}) \ar[d]^{F(\gamma^C(c_1, c_2))} \ar[r]^{\lambda_F(c_1, c_2) \ \ \ }  &\TensP{F(c_1)}{F(c_2)}{D} \ar[d]^{\gamma^D(F(c_1), F(c_2))} \ar[r]^{\TensP{\alpha(c_1)}{\alpha(c_2)}{D}} & \TensP{G(c_1)}{G(c_2)}{D} \ar[d]^{\gamma^D(G(c_1), G(c_2))}  \\
 		G(\TensP{c_2}{c_1}{C}) \ar[r]_{\inv{\alpha}(\TensP{c_2}{c_1}{C})} & F(\TensP{c_2}{c_1}{C}) \ar[r]_{\lambda_F(c_2, c_1) \ \ \ }  &\TensP{F(c_2)}{F(c_1)}{D} \ar[r]_{\TensP{\alpha(c_2)}{\alpha(c_1)}{D}} & \TensP{G(c_2)}{G(c_1)}{D}
 	}
 	\end{equation*}
 	The condition \cite[Defn. 2.4 OL. 3]{Sharma} follows from the following equalities
 	\begin{multline*}
 	\alpha_D(G(c_1), G(c_2), G(c_3)) \circ  \TensP{\lambda_G(c_1, c_2)}{id_{G(c_3)}}{D} \circ \lambda_G(\TensP{c_1}{c_2}{C}, c_3) = \\
 	\TensP{(\TensP{\alpha(c_1)}{\alpha(c_2)}{D})}{\alpha(c_3)}{D}
 	\circ
 	\alpha_D(F(c_1), F(c_2), F(c_3)) \circ 
 	\TensP{\lambda_F(c_1, c_2)}{id_{F(c_3)}}{D} \circ \\\lambda_F(\TensP{c_1}{c_2}{C}, c_3) \circ
 	\inv{\alpha}(\TensP{(\TensP{c_1}{c_2}{C})}{c_3}{C}) = \\
 	\TensP{(\TensP{\alpha(c_1)}{\alpha(c_2)}{D})}{\alpha(c_3)}{D}
 	\circ \TensP{id_{F(c_1)}}{\lambda_F(c_1, c_2)}{D} \circ \lambda_F(c_1, \TensP{c_2}{c_3}{C}) \circ F(\alpha_C(c_1, c_2, c_3)) \\
 	\circ  \inv{\alpha}(\TensP{(\TensP{c_1}{c_2}{C})}{c_3}{C}) = \\
 	\TensP{id_{G(c_1)}}{\lambda_G(c_1, c_2)}{D} \circ \lambda_G(c_1, \TensP{c_2}{c_3}{C}) \circ G(\alpha_C(c_1, c_2, c_3)).
 	\end{multline*}
 	
 	If $F= (F, \lambda_F)$ is a symmetric monoidal functor then so is $G= (G, \lambda_G)$ because \eqref{def-OL-str-G} is a natural isomorphism.
 	
 \end{proof}

The section $S \vee \F(C;V)$ provides us with a unital symmetric monoidal functor $i_{A,V} \circ (S \vee \F(C; V)):B \vee \F(C; V) \to C$ which we denote by $S^\F$. The unital symmetric monoidal functor $S^\F$ has the following Gabriel factorization:
\begin{equation*}
\xymatrix{
	B \vee \F(C; V) \ar[rr]^{S^\F } \ar[rd]_{\Gamma_{S^\F }} && C  \\
	&  G(S^\F) \ar[ru]_\Delta
}	
\end{equation*}
By lemma \ref{Gab-fact-SM} $(G(S^\F), -\boxdot-, \gamma)$ is a permutative category structure.
Also, by the same lemma, $\Gamma$ is a strict symmetric monoidal functor.

%\begin{rem}
%	\label{str-on-cod-inc}
%	By the choice of our natural transformation $\epsilon_S$, see \eqref{str-on-codom-nat-iso}, it follows that
%	\[
%	\delta^C(S^\F(z)) = S^\F(z),
%	\]
%	for all $z \in B \vee \F(C;V)$.
%\end{rem}

\begin{rem}
	\label{SC-SF-comp}
	The following diagram of unital symmetric monoidal functors is commutative:
	\begin{equation*}
	\xymatrix{
		B \vee \F(C;V) \ar[d]_{S \vee \F(C;V)} \ar[r]^{  \ \ \ \ S^\F} & C \\
		A \vee \F(C;V)  \ar[r]_{ \ \ \ \ i_{A,V}} & C \ar[u]_{S_C}
	}
	\end{equation*}
	The above commutative diagram implies that for each object $z \in G(S^\F)$, $\lambda^{S^\F}(z) = \lambda^{S_C}(z)$.
\end{rem}
 We claim that there exists a strict symmetric monoidal functor $P:C \to G(S^\F)$ such that the following diagram, in $\PCat$, is coCartesian:
\begin{equation}
\label{coCart-Sq}
\xymatrix{
	A \ar[d]_G \ar[r]^{i_A} & C \ar[d]^P \\
	B \ar[r]_{\Gamma \ \ \ \ } & G(S^\F) 
}
\end{equation}
where $\Gamma = \Gamma_{S^\F } \circ \iota_B$.
The object function of the functor $P$ is the monoid homomorphism 
\[
Ob(G) \vee V:Ob(A) \vee V \to Ob(B) \vee V.
\]
For any pair of objects $c_1, c_2 \in Ob(C)$, we observe the following equality:
\[
G(S^\F)(P(c_1), P(c_2)) = C(S_C(c_1), S_C(c_2)).
\]
 Now we define the morphism function of $P$ as follows:
 \[
 P(f) := S_C(f),
 \]
 where $f$ is a morphism in $C$. The functoriality of $P$ follows from that of $S_C$.
%The natural isomorphism $\delta = (i_A \vee i_{\F(C;V)}) \circ (\epsilon_S \vee \F(C;V))$ gives us a pair of isomorphisms $\delta_{c_1}:S^\F(P(c_1)) \cong c_1$ and $\delta_{c_2}:S^\F(P(c_2)) \cong c_2$, where $(i_A \vee i_{\F(C;V)}):A \vee \F(C;V) \to C$ is the functor induced by $i_A$ and the inclusion of $\F(C;V)$ into $C$. We define the morphism function of $P$ as the following isomorphism induced by the two aforementioned (natural) isomorphisms:
%\[
%P_{c_1, c_2}:C(c_1, c_2) \cong C(S^\F P(c_1), S^\F P(c_2)) = G(S^\F)(P(c_1), P(c_2)).
%\]

The object function of $P$ is a monoid homomorphism therefore $P(c_1 \underset{C} \otimes c_2) = P(c_1) \boxdot P(c_2)$, for each pair of objects $c_1, c_2 \in Ob(C)$. The following commutative diagram shows that $P(f_1 \underset{C} \otimes f_2) = P(f_1) \boxdot P(f_2)$, for each pair of maps $(f_1, f_2) \in C(c_1, c_2) \times C(c_3, c_4)$:
\begin{equation*}
\xymatrix@C=18mm{
	c_1 \underset{C} \otimes c_2 \ar[r]^{f_1 \underset{C} \otimes f_2} \ar[d]_\delta^\cong  & c_3 \underset{C} \otimes  c_4 \ar[d]^\delta_\cong \\
	S_C(c_1 \underset{C} \otimes  c_2) \ar[r]_{P(f_1 \underset{C} \otimes f_2)} \ar@{=}[d] & S_C(c_3 \underset{C} \otimes  c_4) \ar@{=}[d] \\
	S^\F (P(c_1) \boxdot  P(c_2)) \ar[r]_{P(f_1) \boxdot P(f_2)} \ar[d]_{\lambda_{S^\F}} & S^\F (P(c_3)  \boxdot  P(c_4)) \ar[d]^{\lambda_{S^\F}} \\
	S^\F (P(c_1)) \underset{C} \otimes  S^\F (P(c_2)) \ar[r]_{P(f_1) \underset{C} \otimes P(f_2)}  & S^\F (P(c_3)) \underset{C} \otimes  S^\F (P(c_4))
}
\end{equation*}
Thus, we have defined a strict symmetric monoidal functor $P$ which is fully faithful. Further, each object of $G(S^\F)$ is isomorphic to one in the image of $P$. Thus,discussion $P$ is an equivalence of categories.

\begin{prop}
	The commutative square \eqref{coCart-Sq} is coCartesian.
\end{prop}
\begin{proof}
	In order to show that \eqref{coCart-Sq} is coCartesian, it is sufficient to show that the following commutative square is coCartesian, in light of factorization \eqref{fact-inc} and remark \ref{co-Cart-1}:
	\begin{equation*}
	\xymatrix{
		A \vee \F(C;V) \ar[d]_{G \vee \F(C;V)} \ar[r]^{ \ \ \ i_{A, V}} & C \ar[d]^P  \\
		B \vee \F(C;V) \ar[r]_{\ \ \ \Gamma_{S^\F}  }   & G(S^\F) 
	}
	\end{equation*}
	We will show that whenever we have the following (outer) commutative diagram, there exists a unique dotted arrow $L$ which makes the whole diagram commutative in $\PCat$:
	\begin{equation*}
	\xymatrix{
		A \vee \F(C;V) \ar[d]_{G \vee \F(C;V)} \ar[r]^{ \ \ \ \ i_{A, V}} & C \ar[d]^P \ar@/^1pc/[rdd]^R \\
		B \vee \F(C;V) \ar[r]_{ \ \ \ \ \Gamma_{S^\F}} \ar@/_1pc/[rrd]_T  & G(S^\F) \ar@{-->}[rd]^L \\
		&& X
	}
	\end{equation*}
	Since $Ob(\Gamma_{S^\F})$ is the identity, therefore the object homomorphism $Ob(L)$ has to be the same as $Ob(T)$ in order to make the diagram commutative, therefore we define $Ob(L) = Ob(T)$.
	The morphism function of $L$ is defined as follows:
	\[
	L_{z_1, z_2} := R_{S^\F(z_1), S^\F(z_2)}:G(S^\F)(z_1, z_2) = C(S^\F(z_1), S^\F(z_2)) \to X(L(z_1), L(z_2)).
	\]
	for each pair of objects $z_1, z_2 \in Ob(G(S^\F))$. This defines a functor $L$ which makes the diagram above commutative (in $\Cat$).
	In order to verify that $L$ is a strict symmetric monoidal functor, it is sufficient to show that for each pair of maps $f_1:z_1 \to z_2$, $f_2:z_3 \to z_4$ in $G(S^\F)$, 
	\begin{equation}
	\label{str-mon-L}
	L(f_1 \boxdot f_2) = L(f_1) \underset{X} \otimes L(f_2) = R(f_1) \underset{X} \otimes R(f_2).
	\end{equation}
	We recall that the map $f_1 \boxdot f_2$ is defined by the following commutative diagram:
	\[
	\xymatrix{
	S^\F(z_1) \underset{C} \otimes S^\F(z_3) \ar[r]^{f_1 \underset{C} \otimes f_2}  &  S^\F(z_2) \underset{C} \otimes S^\F(z_4)  \\
	S^\F(z_1 \underset{B} \otimes z_2) \ar[r]_{f_1 \boxdot f_2} \ar[u]^{\lambda^{S^\F}} & S^\F(z_3 \underset{B} \otimes z_4) \ar[u]_{\lambda^{S^\F}}
     }
	\]
	Since $R$ is a strict symmetric monoidal functor, therefore $R(f_1 \underset{C} \otimes f_2) = R(f_1) \underset{X} \otimes R(f_2)$. Now it sufficient to show that $R\lambda^{S^\F} = id$, in order to establish the equalities in \eqref{str-mon-L}.
	We observe that $\lambda^{S^\F} = i_{A, V}(\lambda^{S \vee \F(C;V)})$. Since $G\vee \F(C, V)$ is an acyclic fibration, it follows from remark \ref{lambda-to-id} that $G\vee \F(C, V) \lambda^{S \vee \F(C;V)} = id$. Since $T \circ G\vee \F(C, V) = R \circ i_{A, V}$, it follows that $R(\lambda^{S^\F}) = id$.
%	Let $f_1:S^\F(z_1) \to S^\F(z_2)$ and $f_2:S^\F(z_3) \to S^\F(z_4)$ be two maps in $G(S^\F)$. We want to show that $L(f_1 \boxdot f_2) = L(f_1) \underset{X} \otimes L(f_2)$. We observe the following commutative diagram in $C$:
%	\[
%	\xymatrix{
%	S^F(z_1) \underset{C} \otimes S^F(z_3) \ar[r] \ar[d] & S^F(z_2) \underset{C} \otimes S^F(z_4) \ar[d] \\
%	S_C(S^F(z_1) \underset{C} \otimes S^F(z_3)) \ar[r] & S_C(S^F(z_3) \underset{C} \otimes S^F(z_4))
%    }
%	\]
%	
%	By remark \ref{} $L$ is a strict symmetric monoidal functor.
%	
%	\dots
%	
	The uniqueness of the object functor of $L$ is obvious. The uniqueness of the morphism homomorphism of $L$ can be easily checked.
	
\end{proof}

The main objective of this section is to show that the natural model category $\PCat$ is left proper. The next lemma serves as a first step in proving the main result. The lemma follows from the above discussion:
\begin{lem}
	In the natural model category $\PCat$ a pushout  of a weak-equivalence along a free cofibration is a weak-equivalence.
\end{lem} 
\begin{proof}
	In light of the facts that each weak equivalence in a model category can be factored as an acyclic cofibration followed by an acyclic fibration and acyclic cofibrations are closed under cobase change, it is sufficient to see that the cobase change of an acyclic fibration is a weak-equivalence. This follows from the dicussion above.
	
\end{proof}

Now we state and prove the main result of this note:
\begin{thm}
	\label{left-Prop-Perm}
	The natural model category of permutative categories $\PCat$ is a left proper model category.
\end{thm}
\begin{proof}
	We will show that a pushout $P(F;q)$ of a weak equivalence $F:A \to D$ in $\PCat$ along a cofibration $q:A \to B$ in $\PCat$ is a weak-equivalence. We consider the following commutative diagram:
	
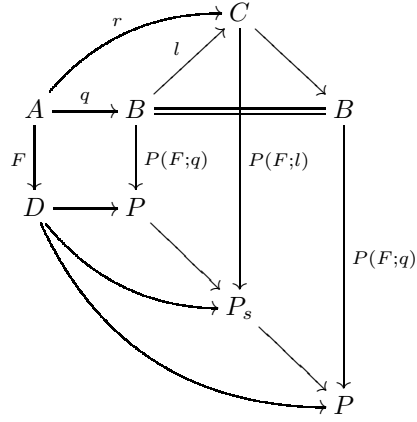
\begin{equation*}
	\xymatrix{
		& & C \ar[ddd]^{P(F;l)}  \ar[rd]\\
		A \ar[r]^q \ar[d]_F \ar@/^1pc/[rru]^r & B \ar[ru]^l \ar[d]^{P(F;q)} \ar@{=}[rr] &     & B \ar[ddd]^{P(F;q)} \\
		D \ar[r] \ar@/_1pc/[rrd] \ar@/_2pc/[rrrdd] & P \ar[rd] \\
		&& P_s \ar[rd] \\
		&&& P
	}
	\end{equation*}
	Since $F$ is a cofibration therefore by proposition \ref{char-cof} there exists a free cofibration $r:A \to C$ such that $F$ is a retract of $r$ by a map that fixes $A$.
	The top left commutative square in the above diagram is coCartesian. The map $P(F;l)$ is a pushout of $F$ along the free cofibration $r$ and therefore a weak-equivalence by lemma \ref{left-Prop-Perm}. Now the result follows from the observation that the diagonal composite $P \to P_s \to P$, in the above diagram, is the identity map and the commutativity of the above diagram.
	
\end{proof}

\appendix
 \section{Gabriel Factorization of symmetric monoidal functors}
\label{Gab-Fact-Perm}

In this appendix we construct a \emph{Gabriel Factorization} of a unital symmetric monoidal functor beyween permutative categories. Our construction factors a unital symmetric monoidal functor into an essentially surjective strict symmetric monoidal functor followed by a fully-faithful unital symmetric monoidal functor.
\begin{lem}
	\label{Gab-fact-SM}
	Each unital symmetric monoidal functor $F:C \to D$ between permutative categories can be factored as follows:
	\begin{equation*}
	\xymatrix{
		C \ar[rr]^{F \ \ \ \ } \ar[rd]_{\Gamma_F} && D  \\
		&  G(F) \ar[ru]_{\Delta_F}
	}	
	\end{equation*} 
	where $\Gamma_F$ is a strict symmetric monoidal functor which is identity on objects and $\Delta$ is fully-faithful.
\end{lem}
\begin{proof}
	We begin by defining the permutative category $G(F)$. The object monoid of $G(F)$ is the same as $Ob(C)$. For a pair of objects $c_1, c_2 \in Ob(C)$, we define
	\[
	G(F)(c_1, c_2) := C(F(c_1), F(c_2)).
	\]
	The Gabriel factorization of the underlying functor of $F$ gives us the following factorization in $\Cat$:
	\begin{equation*}
	\xymatrix{
		C \ar[rr]^{F \ \ \ \ } \ar[rd]_{\Gamma_F} && D  \\
		&  G(F) \ar[ru]_{\Delta_F}
	}	
	\end{equation*}
	We will show that the functor $\Gamma_F$ is strict symmetric monoidal and $\Delta_F$ is unital symmetric monoidal.
	We define a symmetric monoidal structure on $G(F)$ next which we denote by $(G(F),\boxdot, \gamma)$. For any pair of objects $c_1, c_2 \in Ob(G(F))$, we define $c_1\boxdot c_2 := c_1 \underset{C} \otimes c_2$. For a pair of maps $f_1:c_1 \to c_3$ and $f_2:c_2 \to c_4$, we define $f_1 \boxdot f_2$ to be the following arrow:
	\[
	\xymatrix{
		F(c_1 \underset{C} \otimes c_2) \ar[d]_{\lambda_F}^\cong \ar[r]^{f_1 \boxdot f_2} & F(c_3 \underset{C} \otimes c_4) \ar[d]^{\lambda_F}_\cong \\
		F(c_1) \underset{D} \otimes F(c_2) \ar[r]_{f_1 \underset{D} \otimes f_2} & F(c_3) \underset{D} \otimes F(c_4)
	}
	\]
	It is easy to establish that $-\boxdot-$ is a bifunctor: Let $f_3:c_3 \to c_5$ and $f_4:c_4 \to c_6$ be another pair of arrows in $G(F)$. Now we consider the following commutative diagram:
	\[
	\xymatrix{
		F(c_1 \underset{C} \otimes c_2) \ar[d]_{\lambda_F}^\cong \ar[r]^{f_1 \boxdot f_2} & F(c_3 \underset{C} \otimes c_4) \ar[d]^{\lambda_F}_\cong  \ar[r]^{f_3 \boxdot f_4} & F(c_5 \underset{C} \otimes c_6) \ar[d]^{\lambda_F}_\cong \\
		F(c_1) \underset{D} \otimes F(c_2) \ar[r]_{f_1 \underset{D} \otimes f_2} & F(c_3) \underset{D} \otimes F(c_4) \ar[r]_{f_3 \underset{D} \otimes f_4} & F(c_5) \underset{D} \otimes F(c_6)
	}
	\]
	The above diagram tells us that:
	\[
	(f_3 \boxdot f_4) \circ (f_1 \boxdot f_2) = (f_3 \circ f_1) \boxdot (f_4 \circ f_2)
	\]
	because the composite map in the bottom row of the above diagram namely $(f_3 \underset{D} \otimes f_4) \circ (f_1 \underset{D} \otimes f_2)$ is the same as $(f_3 \circ f_1) \underset{D} \otimes(f_4 \circ f_2)$.
	The tensor product $-\boxdot-$ on $G(F)$ is strictly associative because the object set of $G(F)$ is a monoid and the tensor  product of morphisms is associative because the tensor product of morphisms in $G(F)$ is inherited from that in $D$ which is strictly associative.
	%	 because for each triple of arrows $f_1, f_2, f_3$ in $G(F)$, we have the following equalities:
	%	\[
	%	(f_1 \boxdot f_2) \boxdot f_3 = F(f_1 \otimes f_2 \otimes f_3) = f_1 \boxdot (f_2 \boxdot f_3).
	%	\]
	The symmetry natural transformation $\gamma$ is defined on objects as follows:
	\[
	\gamma_{c_1, c_2} := F(\gamma^C_{c_1, c_2}).
	\]
	Let $f_1:c_1 \to c_3$ and $f_2:c_2 \to c_4$ be a pair of maps in $G(F)$. The following commutative diagram shows us that $\gamma$ is a natural isomorphism:
	
	\begin{equation*}
	\xymatrix@C = 16 mm{
		F(c_1 \underset{C} \otimes c_2) \ar@/^1.5pc/[rrr]^{\gamma_{c_1, c_2}} \ar[r]_{\lambda_F} \ar[d]_{f_1  \boxtimes f_2}  & F(c_1) \underset{D} \otimes F(c_2) \ar[r]_{\gamma^D_{F(c_1), F(c_2)}} \ar[d]_{f_1 \underset{D} \otimes f_2} & \ar[d]^{f_2 \underset{D} \otimes f_1} F(c_3) \underset{D} \otimes F(c_4) \ar[r]_{\lambda_F} &F(c_2 \underset{C} \otimes c_1) \ar[d]^{f_2 \boxtimes f_1} \\
		F(c_3 \underset{C} \otimes c_4) \ar@/_1.5pc/[rrr]_{\gamma_{c_3, c_4}} \ar[r]^{\lambda_F} & F(c_3) \underset{D} \otimes F(c_4) \ar[r]^{\gamma^D_{F(c_3), F(c_4)}} & F(c_4) \underset{D} \otimes F(c_3) \ar[r]^{\lambda_F}  & F(c_4 \underset{C} \otimes c_3)
	}
	\end{equation*}
	which shows that $\gamma$ is a natural transformation.
	The following equalities verifies the symmetry condition: 
	\[
	\gamma_{c_1, c_2} \circ \gamma_{c_2, c_1}  = F(\gamma^C_{c_1, c_2})  \circ F(\gamma^C_{c_2, c_1} ) = 
	F(\gamma^C_{c_1, c_2} \circ \gamma^C_{c_2, c_1}) = id.
	\]
	This defines a permutative category structure on the category $G(F)$.
	Using the definition of the symmetric monoidal structure on $G(F)$, one can easily check that $\Gamma_F$ is a strict symmetric monoidal functor.
	
\end{proof}
	\bibliographystyle{amsalpha}
	\bibliography{HomCLQCat}
	
\end{document}